\documentclass[11pt]{amsart}
\usepackage[utf8]{inputenc}
\usepackage{fullpage}
\usepackage{hyperref}
\usepackage{amsfonts}
\usepackage{amsmath}
\usepackage{amssymb}
\usepackage{amscd}
\usepackage{graphicx}
\usepackage{enumitem}
\usepackage{latexsym}
\usepackage{color}
\usepackage{epsfig}
\usepackage{amsopn}
\usepackage{mathrsfs}
\usepackage{array}
\usepackage{caption} 
\theoremstyle{plain}
\newtheorem{lemma}{Lemma}[section]
\newtheorem*{theorem*}{Theorem}
\newtheorem*{lemma*}{Lemma}
\newtheorem*{proposition*}{Proposition}
\newtheorem*{corollary*}{Corollary}
\newtheorem*{problem*}{Problem}
\newtheorem{theorem}[lemma]{Theorem}

\newtheorem{proposition}[lemma]{Proposition}
\theoremstyle{definition}

\newtheorem{remark}[lemma]{Remark}
\newtheorem{observation}[lemma]{Observation}

\newcommand{\R}{\mathbb{R}}

\newcommand{\Q}{\mathbb{Q}}
\newcommand{\Z}{\mathbb{Z} }
\newcommand{\C}{\mathbb{C}}
\renewcommand{\P}{\mathbb{P}}

\newcommand{\A}{\mathcal{A}}
\newcommand{\Pic}{\textrm{Pic}}

\newcommand{\PGL}{\textrm{PGL}}

\begin{document}

\title{The icosahedral line configuration and Waldschmidt
constants
}
\author{Sebastian Calvo}
\address{Department of Mathematics, The Pennsylvania State University, University Park, PA 16802}
\email{calvos1@psu.edu}

\date{\today} 
\thanks{During the preparation of this article the author was partially supported by the NSF FRG grant DMS-1664303 and NSF DMS-2142966.}
\maketitle

\begin{abstract}
There is a highly special point configuration in $\mathbb{P}^2$ of 31 points, naturally arising from the geometry of the icosahedron. The 15 planes of symmetry of the icosahedron projectivize to 15 lines in $\mathbb{P}^2$, whose points of intersections yield the 31 points. Each point corresponds to an opposite pair of vertices, faces or edges of the icosahedron. The symmetry group of the icosahedron is  $G=A_5\times \mathbb{Z}_2$, one of finitely many exceptional complex reflection groups. The action of $G$ on the icosahedron descends onto an action on the line configuration. We blow up $\mathbb{P}^2$ at the 31 points to study the line configuration.  The Waldschmidt constant is 
a measure of how special a collection of points in $\P^2$. In this paper, we study negative $G$-invariant curves on this blow-up in order to compute the Waldschmidt constant of the ideal of the $31$ singularities. 
\end{abstract}

\section{Introduction}

 In this paper, we study a line configuration $\A$ of $15$ lines and $31$ points in $\P^2=\P^2_\C$ that naturally arise from studying the symmetries of the icosahedron, the platonic solid studied heavily by Klein \cite{Klein}. The icosahedron has 15 mirror planes that projectivize to the 15 lines of $\A$. The 15 lines intersect at 6 quintuple points, 10 triple points and 15 double points. Each pair of opposite vertices, faces and edges correspond to a quintuple, triple and double point respectively. 

We may consider $I\subseteq S=\C[x,y,z]$ to be the ideal of a reduced collection of points in $\P^2$.  Define the $m$\textit{-th symbolic power} $I^{(m)}$ to be $\cap_i I_{p_i}^m$ \cite[pg 6]{BrianAsymptotics}. Geometrically, this ideal corresponds to curves having multiplicity at least $m$ at each $p_i$. The \textit{Waldschmidt constant} of an ideal $I$ is defined to be
\[ \widehat{\alpha}(I)=\lim_{m\rightarrow \infty} \frac{\alpha(I^{(m)})}{m} \]
where $\alpha(I)$ is the least positive integer $t$ such that the graded piece $I_t \ne 0.$ In this paper, we compute the Waldschmidt constant for $I_\A$, the homogeneous ideal of singularities of $\A$. The singularities are highly non-general points of $\P^2$. We demonstrate similar techniques to those in \cite{NCSB} in order to calculate $\widehat{\alpha}(I_\A)$. The following is the main result of the paper.

\begin{theorem*}[Main Result]\label{theorem-main} Let $I_\mathcal{A}$ be the homogeneous ideal of singularities of $\A$. Then \[ \widehat{\alpha}(I_\mathcal{A}) = \frac{11}{2}. \] \end{theorem*}

Determining the Waldschmidt constant of a collection of points is difficult in general. The Nagata conjecture \cite{Nagata} states that for $s\ge 10$ very general points in $\P^2$, a curve $C$ passing through each $p_i$ with multiplicity at least $m_i$ satisfies 
\[ \deg C \geq \frac{1}{\sqrt{s}}\sum_{i=1}^s m_i. \]
\noindent The formulation of Nagata's conjecture in terms of Waldschmidt constants states $\widehat{\alpha}(I) = \sqrt{s}.$ From this perspective, the Waldschmidt constant measures how special an arrangement of points in $\P^2$ is.  The Waldschmidt constant can be regarded as the reciprocal of a multi-point Seshadri constant, a well-known measure of local positivity of a line bundle. It is always true that $1\le \widehat{\alpha}(I) \le \sqrt{s}$. For special configurations of points, such calculations of $\widehat{\alpha}(I)$ are non-trivial but we expect  $\widehat{\alpha}(I)$ to be smaller than $\sqrt{s}$. We exploit a relationship between the 31 points and the group $A_5\times \Z_2$ to alleviate the difficulty.

\subsection{ Emphasis of the group action on the point configuration } 

The $60$ rotational symmetries of the icosahedron form a group isomorphic to $A_5$. The full symmetry group  of the icosahedron $A_5\times \Z_2 $ has order $120$, obtained by introducing the reflection through the center of the icosahedron. This group is one of finitely many exceptional complex reflection groups.
For a $(n+1)$-dimensional complex vector space $V$, a \textit{complex reflection group} $G\subseteq \textrm{GL}(V)$ is a finite group that is generated by pseudoreflections: those elements that fix a hyperplane pointwise. These groups were studied and classified in \cite{shephard_todd_1954}. Considering the collection of hyperplanes, given by pseudoreflections of $G$, and projectivizing $V$ yields a hyperplane configuration in $\P^n$. In addition, the natural action of $G$ on $\P^n$ permutes the hyperplanes. Configurations admitted by complex reflection groups exhibit interesting characteristics \cite{drabkin2021singular}. For $n=2$, we obtain a line configuration in $\P^2$. Line configurations have been studied extensively \cite{BrianAsymptotics,orlik2013arrangements} yet continue to be subject of recent research \cite{NCSB,Dolgachev}. 

The group $G$ then naturally acts on the homogeneous coordinate ring $S$ of $\P^2$. An important characteristic of comflex reflection groups is that the invariant ring of such groups is finitely generated \cite{derksen2015computational}. Let $G=A_5\times \Z_2$ The invariant ring $S^G$ is generated as a $\C$-algebra by homogeneous forms $\phi_2,\phi_6,\phi_{10}$ where the subscript denotes the degree of the form. The polynomials $\phi_d$ turn out to be vital in considering curves vanishing to points of the point configuration. This will allow us to construct curves of  high multiplicities at the singularities of $\A$ in order to bound $\widehat{\alpha}(I_\A)$ from above.
 
\subsection{The blow-up of $\P^2$ at the 31 singularities. }

To give a lower bound of $\widehat{\alpha}(I_\A)$, we consider the Picard group of the blow-up $X_\A$ of $\P^2$ at the $31$ singularities and we show certain divisor classes are not effective. This is equivalent to showing certain divisor classes are nef. In addition, the action of $G$ on $\P^2$ extends to Pic$(X_\A)$. We will see that the divisor class $40H-5E_5-7E_3-8E_2$ is $G$-invariant and nef, where $H$ is the divisor class of a line and $E_k$ is the sum of exceptional divisors over the points of multiplicity $k$ of $\A$.


\subsection{Outline} 

In section \ref{s-prelims}, we review group theoretic facts and some representation theory of $G$. We construct the corresponding line configuration $\mathcal{A}$ while emphasizing its geometry with respect to the group action on $\P^2$. In Section \ref{s-wald}, we exhibit an upper bound for $\widehat{\alpha}(I_\A)$. We assume nefness of $D$ to give a lower bound.  In Section \ref{s-meat}, we focus on studying $G$-invariant curves on $X_\A$ in order to prove $D$ is nef. Lastly, in Section \ref{section-sub-point} we compute the Waldschmidt constants of sub point-configurations naturally found in the configuration $\A$. 

\section{Preliminaries}\label{s-prelims}

\subsection{The icosahedron}\label{ss-ico}

We return to the icosahedron. There are $15$ mirror planes of the icosahedron in $\R^3$. On the icosahedron, five planes meet at each vertex, three planes meet at the center of each face and two planes meet at the mid-point of each edge. Figure \ref{fig-markedico-lineconfig} displays the marked icosahedron alongside the line configuration $\A$ in $\P^2$. To obtain $\A$, fix a pair of opposite vertices $\{p,-p\}$ and identify the plane $P$ wedged between the pair of vertices that slices the icosahedron in half.  Project the mirror planes onto the plane containing $p$ parallel to $P$. These are the lines of $\A$ and they meet at $6$ quintuple points, $10$ triple points and $15$ double points. We say that $\A$ is the line configuration corresponding to the group $G.$ 

Alternatively, one may construct $\A$ in $\R^2$ as follows. Fix a regular pentagon $R$. Extend the edges of $R$ to lines. Draw the diagonals of $R$ and line segments between the center of $R$ and each vertex\footnote{Instead of connecting the center to a vertex of $R$, one may also draw a line segment from the center to each midpoint of a side of $R$. This is known to be called an \textit{apothem}. }. Extend these to lines.  These are the $15$ lines of $\A$. In the following subsection, we construct the line configuration algebraically. 

\begin{figure}
    \centering
    \includegraphics[width=8cm]{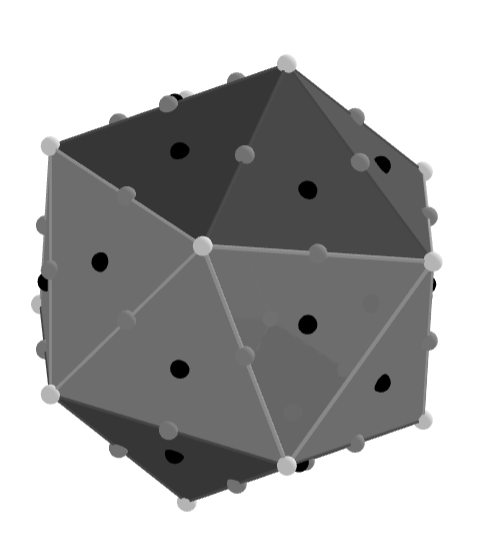}
    \includegraphics[width=8cm]{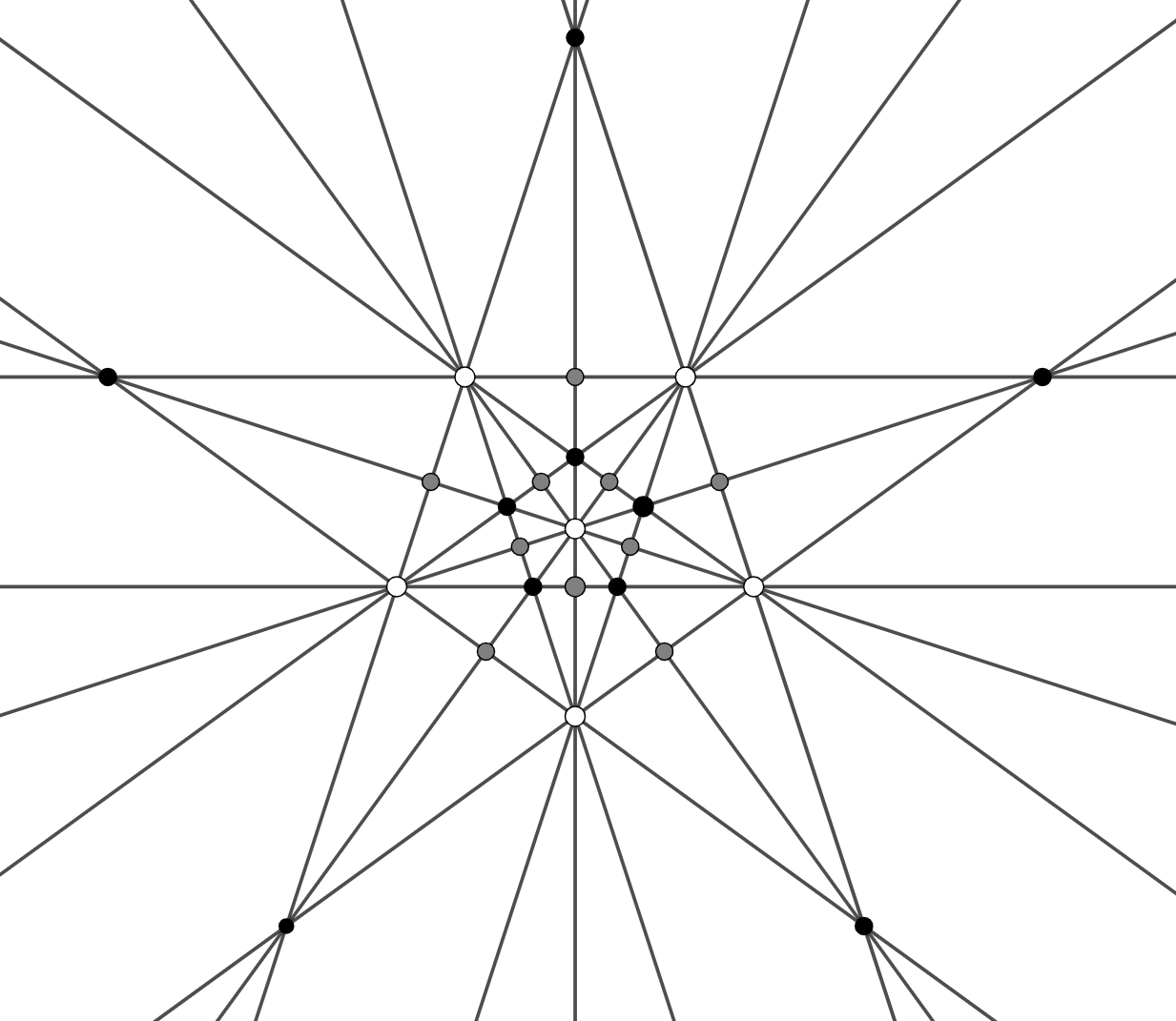}
    \caption{An icosahedron marked with quintuple (white), triple (black) and double (grey) points alongside the line configuration in the affine patch $y-\omega z\ne 0$.  Note that $5$ double points lie on the line at infinity.}
    \label{fig-markedico-lineconfig}
\end{figure}

\subsection{The group $A_5\times \Z_2$} The group $G$ is generated by idempotent elements $g,h,i$. Let $\omega$ be the golden ratio. There is a 3-dimensional representation $\rho$ given by 
    \[  \rho(g) = \begin{pmatrix} -1 & 0 & 0 \\ 0 & 1 & 0 \\ 0 & 0 & 1 \end{pmatrix}, \quad    \rho(h) = \begin{pmatrix} 1 & 0 & 0  \\ 0 & -1 & 0 \\ 0 & 0 & 1 \end{pmatrix},  \quad  \rho(i) = \frac{-1}{2}\begin{pmatrix} \omega-1 & \omega & 1  \\ \omega & -1 & \omega-1 \\ 1 & \omega-1 & -\omega \end{pmatrix}. \]
with the following character values 
\[
\begin{array}{c|rrrrrrrrrr}
  \rm class&\rm1&\rm2A&\rm2B&\rm2C&\rm3&\rm5A&\rm5B&\rm6&\rm10A&\rm10B\cr
  \rm size&1&1&15&15&20&12&12&20&12&12\cr
\hline

  \chi_\rho &3&-3&1&-1&0&\omega&1-\omega&0&\omega-1&-\omega\cr 

\end{array}.
\]
There are 15 pseudoreflections as observed from the character values of $\rho$. Indeed there $15$ elements of conjugacy class $2B$ whose trace is $1$. This implies for each such element, the dimension of the eigenspace $E_{1}$ is $2$. This is an equivalent condition to fixing a hyperplane pointwise. Now considering the  arrangement of these $15$ hyperplanes, the projectivization of the arrangement in $\P^2$ admits $\A$. Let $\widetilde{G}=A_5$, the image of $G$ in $\PGL(3,\C)$. 

\begin{remark}\label{orbitstabs} The action of $\widetilde{G}$ on $\P^2$ with respect to the line configuration is the following. Notice that the lines and points may be defined over $\Q(\omega)$. 
\begin{enumerate}
    \item The line $x=0$ is a line of $\A$. The 15 lines of $\A$ form a single orbit of size $15$. 
    \item The point $p_2=[0:0:1]$ is a double point and has stabilizer $\widetilde{G}_{p_2}=D_4$, the Klein four-group. The double points of $\A$ form a single orbit of size $15$ and each double point has stabilizer isomorphic to $D_4$. 
    \item The point $p_3=[1:1:1]$ is a triple point and has stabilizer $\widetilde{G}_{p_3}=D_6.$ The triple points of $\A$ form a single orbit of size $10$ and each triple point has stabilizer isomorphic to $D_6$. 
    \item The point $p_5=[\omega+1:0:2]$ is a quintuple point and has stabilizer $\widetilde{G}_{p_5}=D_{10}$. The quintuple points of $\A$ form a single orbit of size $6$ and each quintuple point has stabilizer isomorphic to $D_{10}$. 
\end{enumerate}
\end{remark}

The group $G$ acts naturally on the homogeneous coordinate ring $S$ of $\P^2$. We may then consider the ring of invariants $S^G$, the polynomials that are invariant under the action of $G$. The curves given by these polynomials are $G$-invariant. The ring $S^G$ is generated by particular polynomials $\phi_2,\phi_6,$ and $\phi_{10}$ of degrees $2,6,$ and $10$ \cite{Klein}. The invariant ring $S^{\widetilde{G}}$ is generated by these polynomials in addition to a fourth invariant $\phi_{15}$. The polynomials $\phi_2,\phi_6$ and $\phi_{10}$ are algebraically independent, while $\phi_{15}^2$ can be expressed in terms of the other generators. To obtain the generating polynomials $\phi_d$, note that $G$ is an orthogonal group since it is the symmetry group of the regular icosahedron. Since each element of $G$ preserves distance, we have $\phi_2=x^2+y^2+z^2$ to be a degree $2$ invariant. Recall from Section \ref{ss-ico} the pair of opposite vertices and the unique plane $P$ corresponding to this pair. There are $6$ such pairs of opposite vertices, thus $6$ unique planes. We take $\phi_6$ to be the product of the defining equations of these planes 

\[\phi_6=x^{4}y^{2}+y^{4}z^{2}+x^{2}z^{4}+4\,\omega\,x^{2}y^{2}z^{2}-\left(\omega+1\right)(x^{2}y^{4}+y^{2}z^{4}+x^{4}z^{2}).\]

Note that each plane $P$ contains $5$ double points and so for each double point $p_2$, we have $\phi_6\equiv 0  $ modulo $\mathfrak{m}_{p_2}^2$, the maximal ideal corresponding to $p_2$. Similarly, there ten pairs of opposite faces with opposite centers. There is a unique plane sandwiched between a pair of opposite faces that cuts the icosahedron in half. We take $\phi_{10}$ to be the product of the defining equations of the planes 

\begin{align*} \phi_{10}=&x^8y^2+x^2z^8+y^8z^2+(3\omega-5)(x^2y^8+x^8z^2+y^2z^8)+(3\omega-7)(x^6y^4+x^4z^6+y^6z^4)\\
&-(6\omega-11)(x^4y^6+x^6z^4+y^4z^6)-(30\omega-40)(x^6y^2z^2+x^2y^6z^2+x^2y^2z^6)\\
&+(45\omega-60)(x^2y^4z^4+x^4y^2z^4+x^4y^4z^2)
\end{align*}

Similarly, we have $\phi_{10}\equiv 0 \mod \mathfrak{m}_{p_2}^2$ for each double point. The polynomial ${\phi}_{15}$ is the defining polynomial of the line configuration $\A$. It may be obtained as the Jacobian determinant  
\[ \phi_{15} =  \begin{vmatrix}
\partial\phi_2/\partial x & \partial\phi_6/\partial x & \partial \phi_{10}/\partial x\\
\partial\phi_2/\partial y & \partial\phi_6/\partial y & \partial \phi_{10}/\partial y\\
\partial\phi_2/d\partial z & \partial\phi_6 / \partial y & \partial \phi_{10}/\partial z

\end{vmatrix}.
\]
The polynomial $\phi_{15}$ satisfies the following relation 
\begin{align*}
c\phi_{15}^2=&125\phi_{10}^3+(1300\omega-2275)\phi_2^2\phi_6\phi_{10}^2+(-12\omega+16)\phi_2^5\phi_{10}^2+(46800\omega-75600)\phi_2\phi_6^3\phi_{10}\\
&-(6360\omega-10335)\phi_2^4\phi_6^2\phi_{10}+(200\omega-320)\phi_2^7\phi_6\phi_{10}+(343872\omega-556416)\phi_2^5\\
&-(84624\omega-136912)\phi_2^3\phi_6^4+(6916\omega-11193)\phi_2^6\phi_6^3-(188\omega-304)\phi_2^9\phi_6^2
\end{align*}
for an appropriate constant $c\in \C^*$.

\begin{remark}\label{chopped} Continuing the theme of Remark \ref{orbitstabs}, we record how $\widetilde{G}$ acts on orbits of $\P^2$. The representation $\rho$ is irreducible and so $G$ fixes no point of $\P^2$. 
\begin{enumerate}
    \item The double points of $\mathcal{A}$ form an orbit of $15$ points.
    \item The triple points of $\mathcal{A}$ form an orbit of 10 points.
    \item The quintuple points of $\mathcal{A}$ form an orbit of $6$ points.
    \item The curves $\phi_2=0$ and $\phi_6=0$ intersect at and form an orbit of $12$ points. These points do not lie on $\A$. One of the points is 
     \[ \left[\sqrt{\frac{1}{2}(5+\sqrt{5}}):1:1-\omega \right] \] 
     and this point is not defined over $\Q(\omega)$.  
    \item The curves $\phi_2=0$ and $\phi_{10}=0$ intersect at and form an orbit of $20$ points. These points do not lie on $\A$. One of these points is 
    \[ [\zeta:\overline{-\zeta}:1]\]
    where $\zeta=e^{2\pi i/3}$. Similarly, one would need to extend to the field $\Q(\omega,\zeta)$ to define these intersections. 
    \item The curves $\phi_6=0$ and $\phi_{10}=0$ intersect at the double points of $\A$. Each curve passes through a double point twice, therefore by Bezout's theorem, the curves only meet at the double points.  
    \item A non-singular point of the line configuration has an orbit of $30$ points.
    \item Otherwise a point has an orbit of $60$ points. 

\end{enumerate}
\end{remark}

\section{Waldschmidt Computations on the blow-up of $\P^2$}\label{s-wald}

In this section, we compute an upper bound for $\widehat{\alpha}(I_\A)$ and discuss our approach to proving a lower bound. To give such bounds, we consider the blow-up of $\P^2$ at the singularities of $\A$. In particular, we study certain divisor classes on this blow-up. These divisor classes correspond to curves on $\P^2$ with incidence relations to $\A$. To obtain a lower bound, we rely on the nefness of a specific divisor $D$. We prove the nefness of $D$ in Section \ref{s-meat}. In Subsection \ref{ss-sub-point}, we consider the 3 point configurations corresponding to the three types of $k$-uple points and directly compute the Waldschmidt constant of these configurations.   

\subsection{Divisor classes on the blow-up} Denote the blow-up of $\P^2$ at the singularities of $\mathcal{A}$ by $X_\mathcal{A}$. Then $\Pic(X_\A)$ is generated by the pullback of a line $H$ and 31 exceptional divisors. The divisor class of the proper transform of the line configuration on the blow-up is 
\[ A = 15H-5E_5-3E_3-2E_2 \]
where $E_m$ is the sum of the exceptional divisors lying over the points of multiplicity $m$ for $m\ge 2$. These divisor classes on the blow-up are subject to the following intersections 
\[ H^2 = 1,\hspace{.6cm} E_5^2 = -6, \hspace{.6cm} E_3^2 = -10, \hspace{.6cm} E_2^2 = -15, \hspace{.2cm} \textrm{ and } \hspace{.2cm}  H\cdot E_i = E_i\cdot E_j = 0 \]
for $i\ne j$. For example, $A^2 = -75$. With this, we now are able to bound $\widehat{\alpha}(I_\A)$ from above.

\begin{lemma} Let $I_\A$ be the homogeneous ideal corresponding to the singularities of $\mathcal{A}$. Then 
\[ \widehat\alpha(I_A)\le \frac{11}{2}. \]
\end{lemma}
\begin{proof} Consider the divisor class 
\[ D = 40H-5E_5-7E_3-8E_2 .\]
For $k\ge 1$, we calculate the expected dimension of  the linear series $|kD + 2H|$ to be 
\[ \chi(kD + 2H) = {40k + 4 \choose 2} - 6{5k+1 \choose 2} - 10{7k+1 \choose 2} - 15{8k+1 \choose 2} = 6+30k > 0. \]
Therefore $kD + 2H + kA=(55k+2)H-10kE_5-10kE_3-10kE_2$ is an effective divisor class, and there is some element in $I^{(10k)}_\A$ of degree $55k+2$. This gives 
\[  \widehat{\alpha}(I_\A)=\lim_{m\rightarrow \infty} \frac{\alpha(I^{(m)})}{m} \le \lim_{k\rightarrow \infty} \frac{55k+2}{10k} = \frac{11}{2}.\qedhere\]
\end{proof}

To give a lower bound for $\widehat{\alpha}(I_\A)$, we assume the nefness of $D$.

\begin{lemma}\label{lemma-11/2} If $D = 40H-5E_5-7E_3-8E_2$ is a nef divisor, then $\widehat{\alpha}(I_\A)=\frac{11}{2}$.
\end{lemma}
\begin{proof}
Assume $D$ is nef and suppose for contradiction there exists a $\beta\in \Q$ such that 
\[ \widehat{\alpha}(I_\A)< \beta < \frac{11}{2}, \]
so that $F = \beta H - E_5-E_3-E_2$ is effective. Then
\[ F \cdot D = 40\beta-30-70-120 = 40\left(\beta-\frac{11}{2}\right) < 0. \]
This contradicts the nefness of $D$. Therefore $\widehat\alpha(I)\ge \frac{11}{2}$. By the preceding lemma, we have equality. \end{proof}

\section{$G$-irreducible and $G$-invariant curves}\label{s-meat}

In this section, we prove the divisor $D$ is nef. Together with Lemma \ref{lemma-11/2}, Theorem \ref{theorem-nef} proves the main result of the paper.

\subsection{A decomposition of the divisor class $D$}\label{ss-decomp} 

A curve $\mathcal{C}$ is \textit{$G$-irreducible} if its irreducible components are in a single orbit under the action of $G$. For example, $G$ acts transitively on the lines of $\A$ and so the curve $\phi_{15}=0$ is $G$-irreducible. A curve $\mathcal{C}$ is \textit{$G$-invariant} if for every $g\in G$, $g(C)=C$. Observe that $G$-irreducibility of a curve implies the curve is $G$-invariant. We first consider the point-configuration of quintuple points.

\begin{remark}\label{remarFk-irreducibility}
The group $G$ acts on $\P^2$. This action extends to a $G$-action on $\Pic(X_\A)$. If a divisor in $\Pic(X_\A)$ is effective, then we ask if there is a $G$-invariant curve in the corresponding linear series. This suggests to we pay attention to curves invariant under the action of $G$. 
\end{remark}

\begin{observation}\label{ob-classes}Consider the following divisor classes \begin{alignat*}{3}
    &A = 15H &-5E_5-3E_3-2E_2,\\
    &B =\hspace{.15cm} 6H & -2E_2,  \\ 
    &C = 30H &-2E_5-6E_3-6E_2,\\
    &D = 40H &-5E_5-7E_3-8E_2.
\end{alignat*}
\begin{enumerate}
    \item The line configuration $\A$ is $G$-irreducible.
    \item A curve $\mathcal{B}$ of divisor class $B$ can be given by the vanishing of $ \phi_6$. The curve $\mathcal{B}$ is a union of $6$ lines (Figure \ref{fig-phi6}), all in a single orbit under the action of $G$. It follows that $\mathcal{B}$ is $G$-invariant.
    \item The divisor classes $A,B,C$ have negative self-intersection and $D^2=0.$
    \item The divisor class $D$ may be written as the linear combination $6D = 4A + 5B + 5C$.
\end{enumerate} 
\end{observation}

\begin{figure}
    \centering
    \includegraphics[width=8cm]{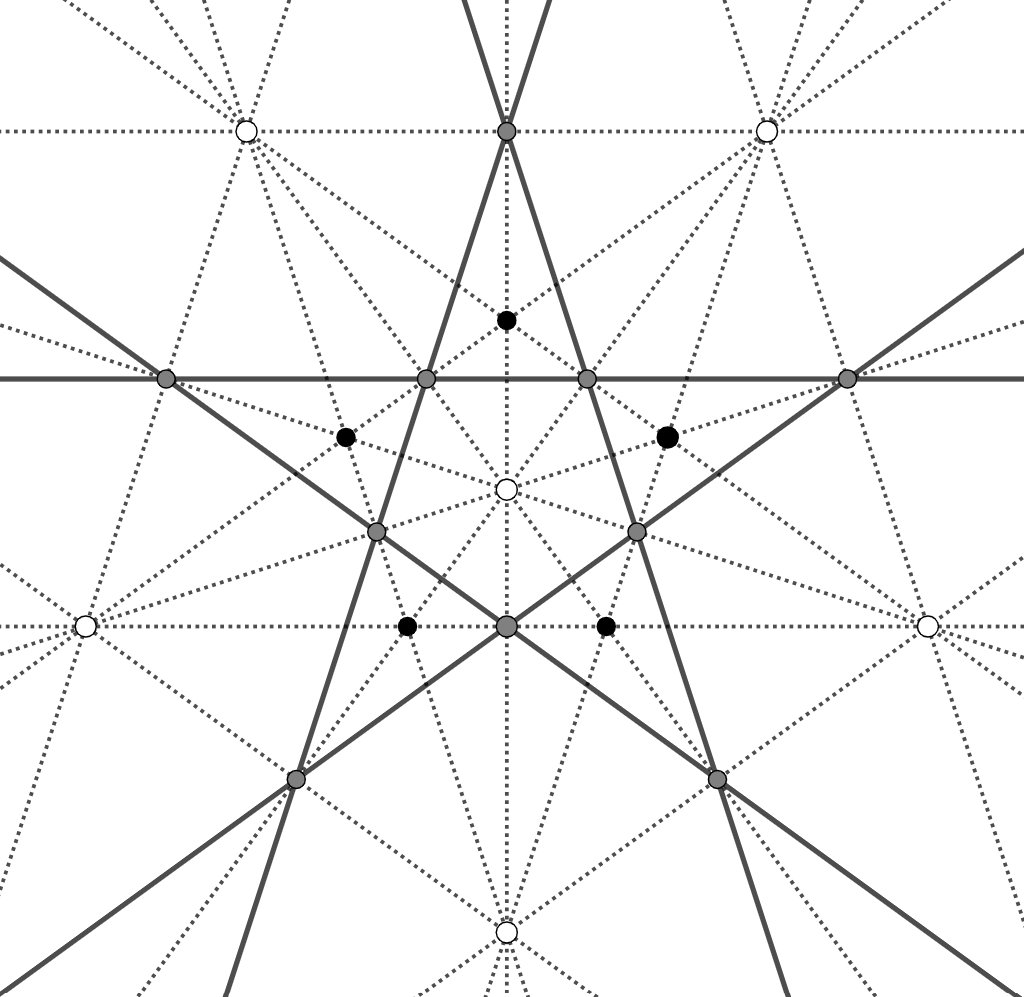}
    \caption{The curve $\phi_6=0$ is a union of the five solid lines and the line at infinity.}
    \label{fig-phi6}
\end{figure}
\begin{proposition}\label{prop-C}
The divisor class $C$ is effective and there is a $G$-irreducible curve $\mathcal{C}$ of class $C$. 
\end{proposition}
We dedicate Section \ref{ss-C} to give an equation for a $G$-irreducible curve $\mathcal{C}$ of divisor class $C$. This proposition allows us to prove the main result.
 
\begin{theorem}\label{theorem-nef} The divisor class $D=40H-5E_5-7E_3-8E_2$ is nef. Consequently 
\[ \widehat{\alpha}(I_\mathcal{A}) = \frac{11}{2}. \] 
\end{theorem}
\begin{proof}
By Observation \ref{ob-classes} and Proposition \ref{prop-C}, the divisor classes $A,B,$ and $C$ are effective and $G$-irreducible. The decomposition $6D = 4A + 5B + 5C$ shows that $D$ is effective.  Since an effective divisor will only intersect non-negatively on irreducible components, it suffices to use the $G$-action on $\Pic(X_\A)$ to show that $D$ intersects non-negatively on its $G$-irreducible components. Since $D\cdot A = D\cdot B = D\cdot C =0$, $D$ is nef. The Waldschmidt constant $\widehat{\alpha}(I_\mathcal{A}) = 11/2$ follows from Lemma \ref{lemma-11/2}.
\end{proof}

\subsection{On $G$-invariant curves}\label{ss-invariant}

To prove Proposition \ref{prop-C}, we produce an equation for the curve $\mathcal{C}$ and prove it is $G$-irreducible. Since $G$-irreducible curves are also $G$-invariant, the following lemma says such defining equations are contained in the subalgebra $T=\C[\phi_2,\phi_6,\phi_{10}]\subseteq S$. Consider the map 
\begin{align*}
    \varphi :\hspace{.1cm}&\P^2\rightarrow \P(2,6,10)\\
    &  p \mapsto [\phi_2(p) : \phi_6(p) : \phi_{10}(p)]
\end{align*}
where $\P(2,6,10)$ is a weighted projective space. This space is isomorphic to the quotient space $\P^2/G$ and $\varphi$ is the quotient map.

\begin{lemma}\label{lemma-defining} Let $\mathcal{C}\subseteq \P^2$ be a $G$-invariant curve which does not contain the line configuration $\A$. Then the defining equation $f\in S$ of $\mathcal{C}$ is $G$-invariant and is contained in $T$. 
\end{lemma}
\begin{proof}
 We recall the invariants $\phi_2$, $\phi_6$ and $\phi_{10}$ from the preliminaries. The map $\varphi$ is a $60$-uple covering away from the line configuration and the finitely many points in $(1)-(6)$ of Remark \ref{chopped}. In particular, the map $\varphi$ is a local isomorphism at a general point $p\in \mathcal{C}$. We have then that $\varphi(\mathcal{C})$ is defined by a weighted homogeneous equation $g(w_0,w_1,w_2)=0$ where $w_0,w_1,w_2$ are the coordinates of $\P(2,6,10)$. Therefore the pullback $\varphi^*g=f$ defines $\mathcal{C}$ and is contained in $T$. 
\end{proof}

This prompts the following definition: for $m_5,m_3,m_2\ge 0$, let 
\[  T_d(-m_5E_5-m_3E_3-m_2E_2)\]
be the finite-dimensional subspace of $T$ of homogeneous degree $d$ forms that are $m_5$-uple at the $6$ quintuple points, $m_3$-uple at the 10 triple points and $m_2$-uple at the $15$ double points of $\mathcal{A}$. Elements of this space yield $G$-invariant curves in the linear series  $|dH -m_5E_5-m_3E_3-m_2E_2|$. Consider the following curves. There is a unique curve given by the vanishing of $\psi_2\in T_2(0)$. Let $\psi_6\in T_6(0)$ be the unique (up to scale) degree $6$ invariant such that the curve $\psi_6=0$ passes through a double point and $\psi_{10}\in T_{10}(0)$ be the unique (up to scale) invariant such that the curve $\psi_{10}=0$ passes through a double point and a triple point. Additionally, there is a unique invariant $\psi_6'\in T_6(0)$ that passes through a triple point. The invariants $\psi_6$,$\psi_6'$ and $\psi_{10}$ enjoy incidence relations with respect to $\A$, as the following lemma shows. 

\begin{lemma} \label{lemma-camille} Let $p,p'\in \A \subseteq \P^2$ be $k$-uple points. If $C$ is a $G$-invariant curve passing through $p$, then $C$ also passes through $p'$.  
\end{lemma}
\begin{proof}
By Remark \ref{chopped}, there is an element $g\in G$ such that $g(p)=p'$. Since $p$ vanishes at $C$, then $p'=g(p)$ vanishes at $C=g(C)$. 
\end{proof}

We take advantage of the vanishings at $k$-uple points of $\psi_d=0$ to give an equation for $\mathcal{C}$. An additional exploit is to observe the action of the stabilizer $\widetilde{G}_p\subseteq \widetilde{G}\subseteq \PGL(3,\C)$ on the local ring $(\mathcal{O}_p, \mathfrak{m}_p)$ at a singularity $p\in \A$.

\begin{lemma}\label{lemma-jacks} Let $p\in \mathcal{A}\subseteq \P^2$ be a singularity and $\widetilde{G}_p\subseteq G$ the stabilizer of $p$.  Let $w$ be a linear form not passing through $p$. If $\psi\in T$ is a $\widetilde{G}$-invariant homogeneous form and vanishes at $p$, then $\psi$ vanishes to order at least $2$ at $p$ and $\widetilde{\psi}=\psi/w^d$ lies in the $1$-dimensional trivial $\widetilde{G}_p$-submodule of $\mathfrak{m}_p^2 / \mathfrak{m}_p^3$.  
\end{lemma}
\begin{proof}
The proof is analogous to \cite[Lemma 4.8]{NCSB} but we describe the upshot in the current setting. From Remark \ref{orbitstabs}, the stabilizer $\widetilde{G}_p$ for a singularity $p\in \A$ is $D_{2n}$ for $n=2,3$ or $5$. 
The key observation used in the proof is that the linear characters of $D_{2n}$ have order either $1$ or $2$, and therefore must act trivially on a tangent line at $p$. 
\end{proof}


Therefore, the previous two lemmas combine to state that if a $G$-invariant curve passes through a $k$-uple point, it passes through every $k$-uple point at least twice. This will allow us to restrict the terms of invariant form vanishing at $p$. 

\subsection{Equation of a curve of divisor class $C$}\label{ss-C}

 For constants $\lambda_i \in \C$, define the invariant \[ \psi_{30} :=  \lambda_1\psi_{10}^3+\lambda_2\psi_2^2\psi_6\psi_{10}^2+\lambda_3\psi_2\psi_6^2\psi_6'\psi_{10}+\lambda_4\psi_6^3\psi_6'^2.\]

For the invariants $\psi_{10}^3, \psi_2^2\psi_6\psi_{10}^2,$ $\psi_2\psi_6^2\psi_6'\psi_{10}$ and $\psi_6^3\psi_6'^2$,  Lemma \ref{lemma-jacks} says each invariant vanishes at order $6$ and $4$ at the double and triple points. To ensure that the curve $\mathcal{C}$ defined by $\psi_{30}$ is of divisor class $C$, we intend to select appropriate scalars $\lambda_i$ so that $\psi_{30}$ vanishes to two more orders at the triple points and twice at the quintuple points.

We now apply Lemma \ref{lemma-jacks} to a triple point $p_3.$ Choose local affine coordinates $\widetilde{x},\widetilde{y}$ centered at $p_3$ so that $\mathfrak{m}_{p_3}^k = (\widetilde{x},\widetilde{y})^k$. Let $w$ be a linear polynomial not vanishing at $p_3$ and denote $\widetilde{\psi}_d:=\psi_d / w^d \in \mathcal{O}_{p_3}$. We have the following decompositions 
\begin{align*}
    \widetilde{\psi}_{10} &\equiv A_2 + A_3 \mod \mathfrak{m}_{p_3}^4\\
    \widetilde{\psi'}_6 &\equiv B_2 + B_3 \mod \mathfrak{m}_{p_3}^4\\
    \widetilde{\psi}_6 &\equiv C_0 + C_1 \mod \mathfrak{m}_{p_3}^2\\
    \widetilde{\psi}_2 &\equiv D_0 + D_1 \mod \mathfrak{m}_{p_3}^2.
 \end{align*}
where $A_i,B_i,C_i,D_i$ are homogeneous polynomials of degree $i$ in $\C[\widetilde{x},\widetilde{y}]$. We therefore have the following relations 
\begin{align*}
    A_2 &= \mu B_2\\
    C_0 &= \nu D_0\\
    C_1 &= 3\nu D_1
\end{align*}
for some $\mu,\nu \in \C^*$, where the third equation is obtained by computing the image of $C_0^3\widetilde{\psi}_2^3 - D_0\widetilde{\psi}_6$ in $\mathfrak{m}_{p_3}$. Therefore we have 
\begin{align*}
    0\equiv C_0^3\widetilde{\psi}_2^3 - D_0\widetilde{\psi}_6 &\equiv D_0^3(C_0+C_1)-C_0(D_0+D_1)^3 \mod \mathfrak{m}_{p_3}^2\\
    &  \equiv D_0^3(C_1-3\nu D_1) \mod \mathfrak{m}_{p_3}^2.
\end{align*} 

Note that $C_0$ and $D_0$ are non-zero since $\psi_2$ and $\psi_6$ do not vanish at a triple point. The constants $\mu,\nu$ and $D_0$ are dependent on the choice of triple point $p_3$ and linear form $w$. Despite this, we may observe that $\alpha :=\nu/\mu\in \C^*$ satisfies
\[ \alpha\psi_2\psi_{10}\equiv \psi_6\psi_6' \mod I_{p_3}^3. \]
This equation is $G$-invariant and thus does not depend on $p_3$ or $w$. The constant of proportionality $\alpha$ is determined by the invariants $\psi_d$.

\begin{lemma}\label{lem-5uple} The curve $\psi_{30}$ is $5$-uple at $p_3$ if 
\[ \lambda_2 + \alpha\lambda_3 + \alpha^2\lambda_4 = 0.\]
\end{lemma}
\begin{proof}
We consider $\widetilde{\psi}_{30}$ modulo $\mathfrak{m}_{p_3}^5$
\begin{align*}
    \widetilde{\psi}_{30} \equiv \mu^2\nu D_0^3 B_2^2\lambda_2 + \mu \nu^2 D_0^3 B_2^2\lambda_3 + \nu^3D_0^3B_2^2\lambda_4 \mod \mathfrak{m}_{p_3}^5 \equiv 0.
\end{align*}
The linear condition follows from factoring out by $D_0^3B_2^2$ and dividing by $\mu^3$. 
\end{proof}

\begin{lemma}\label{lem-6uple} The curve $\psi_{30}$ is $6$-uple at $p_3$ if it is $5$-uple and 
\begin{alignat*}{2}
    5\lambda_2 + 7\alpha\lambda_3 +9\alpha^2\lambda_4  &=0\\
    \lambda_3+\hspace{.15cm}  2\alpha\lambda_4 &=0 \\ 
    2\lambda_2+\hspace{.25cm} \alpha\lambda_3 \hspace{1.4cm} &=0.
\end{alignat*}
\end{lemma}
\begin{proof} Suppose the linear condition of Lemma \ref{lem-5uple} holds. Consider only the degree $5$ terms of $\widetilde{\psi}_{30}$ modulo $\mathfrak{m}_{p_3}^6$.  

\begin{align*}
    \widetilde{\psi}_{30} \equiv &\lambda_2(5\mu^2\nu D_0^2D_1B_2^2+2\mu\nu D_0^3B_2A_3) + \lambda_3(7\mu\nu^2D_0^2D_1B_2^2+\mu\nu^2D_0^3B_2B_3+\nu^2D_0^3B_2A_3)\\
    &+\lambda_4(9\nu^3D_0^2D_1B_2^2+2\nu^3D_0^3B_2B_3) \mod \mathfrak{m}_{p_3}^5.
\end{align*}
Observe that we may group these according to the terms $D_0^2D_1B_2^2,D_0^3B_2A_3$ and $D_0^3B_2B_3$. 
\begin{align*}
    D_0^2D_1B_2^2&(5\mu\nu\lambda_2 + 7\mu\nu^2\lambda_3 + 9\nu^3\lambda_4) =0\\
    D_0^3B_2A_3&(2\mu\nu\lambda_2 + \nu^2\lambda_3) = 0\\
    D_0^3B_2B_3&(\mu\nu^2\lambda_3+2\nu^3\lambda_4) = 0.\qedhere
\end{align*}\end{proof}

We now fix multiples for the invariants $\psi_d$. The choice of multiples is due to experimentation that simplifies the polynomial $\psi_{30}$ as much as possible.
\begin{align*}
    \psi_2 &= 3\phi_2\\
    \psi_6 &= -3(\omega-1)\phi_6\\
    \psi_6'&= -25(\phi_2^3-27(\omega-1)\phi_6)=-25\left( \frac{1}{27}\psi_2^3+9\psi_6 \right) \\
    \psi_{10} &= \frac{15}{4}\left( 25(\omega-1)\phi_2^2\phi_6-(9\omega+3)\phi_{10}\right) .
\end{align*}
Note that for the double and triple points $p_2=[0:0:1]$ and $p_3 = [1:1:1]$, we have 
\begin{align*}
\varphi(p_2)&=[1:0:0] \\
\varphi(p_3)&=\left[3:\omega:45\omega-60\right] \end{align*}
from which we see that the $\psi_d$'s have the required vanishing conditions. 

\begin{lemma} The curve $\psi_{30}$ is double at a quintuple point $p_5$ if and only if 
\[32\lambda_1=9\lambda_2+15\lambda_3+25\lambda_4\]
\end{lemma}
\begin{proof}Select the quintuple point $p_5 = [\omega:0:1]$. It is straightforward to evaluate the terms $\psi_{10}^3, \psi_2^2\psi_6\psi_{10}^2,$ $\psi_2\psi_6^2\psi_6'\psi_{10}$ and $\psi_6^3\psi_6'^2$ at $p_5$ to obtain that $\psi_{30}(p_5)=0$ if and only if 

\[ 32\lambda_1 -9\lambda_2-15\lambda_3-25\lambda_4 = 0.
\qedhere \]
\end{proof}

\noindent Observing that the selection of multiples gives $\alpha=-1$, then the linear system governing these vanishing conditions is the following 
\[ \begin{pmatrix}
-32 & 9 & 15 & 25 \\
0 & 1 & -1 & 1 \\
0 & 5 & -7 & 9\\
0 & 0 & 1 & -2 \\
0 & 2 & -1 & 0 
\end{pmatrix}
\begin{pmatrix}
\lambda_1 \\ \lambda_2 \\ \lambda_3 \\ \lambda_4
\end{pmatrix}
=\begin{pmatrix}
0 \\0\\0\\0
\end{pmatrix}
\]
\begin{theorem}\label{theorem-psi30}
The unique solution to the system above is $(2,1,2,1)$. The invariant $\psi_{30}$ then has the form
\[ \psi_{30} = 2\psi_{10}^3+\psi_2^2\psi_6\psi_{10}^2+2\psi_2\psi_6^2\psi_6'\psi_{10}+\psi_6^3\psi_6'^2.\]
Additionally, the curve $\mathcal{C}$ defined by $\psi_{30}$ is $G$-irreducible.
\end{theorem} 
 \begin{proof} We verify the $G$-irreducibility of the curve $\mathcal{C}$ defined by $\psi_{30}$. It is sufficient to check that the curve $\mathcal{C}'$ in $\P(2,6,10)$ defined by 
\[g: w_{10}^3+w_2^2w_6w_{10}^2+2w_{10}w_{6}^2w_{2}w_6'+w_6^3w_6'^2 \]
is irreducible, where $w_2,w_6,w_{10}$ are the coordinates of $\P(2,6,10)$ and $ w_6'=\frac{1}{27}\psi_2^3+9\psi_6 $. Consider the quotient map giving an isomorphic quotient space to that in Lemma \ref{lemma-defining}
\begin{align*}
    \widetilde{\varphi} :\hspace{.1cm}&\P^2\rightarrow \P(2,6,10)\\
    &  p \mapsto [\psi_2(p) : \psi_6(p) : \psi_{10}(p)].
\end{align*}

The pullback $\widetilde{\varphi}^*\mathcal{C}'$ is $\mathcal{C}$ and so we prove irreducibility in terms of $w_i$.  Firstly, assume $g$ has the  factorization 
\[ (F_2w_{10}^2+F_1w_{10}+F_0)(G_1w_{10}+G_0)\]
where $F_i,G_i\in \C[w_2,w_6]$ are homogeneous polynomials of appropriate degree to make the factors homogeneous. Since $F_2G_1 = 1$, then $G_1$ is constant and $G_0$ must be also be degree $10$. Since $G_0$ divides $w_6^3w_6'^2$, $G_0$ must necessarily divide $w_6'$ or $w_6'^2$. This contradicts the degree polynomial $G_0$. On the other hand, if the factorization of $g$ were of the form 
\[ (F_1w_{10}+F_0)(G_1w_{10}+G_0)(H_1w_{10}+H_0). \]
Then the $F_1,G_1,H_1$ must be constant, forcing $F_0,G_0,H_0$ to be of degree $10$. A similar reasoning as above concludes that the curve $C$ is $G$-irreducible. 
\end{proof}

\section{Exploring the sub point-configurations} \label{section-sub-point}

In this section, we consider the sub point-configurations corresponding to the quintuple, triple and double points of $\A$. 
For $k=2,3$ and $5$, denote $I_k$ to be the intersection of homogeneous ideals of $k$-uple points, In contrast to the Waldschmidt constant of $I_\mathcal{A}$, the Waldschmidt constant of $I_k$ is easier to obtain due to the smaller number of considered points.

\begin{proposition}\label{prop-2}
Let $I_2$ be the homogeneous ideal of double points. Then $\widehat{\alpha}(I_2)=3$. 
\end{proposition}
\begin{proof} Recall the effective divisor $B=6H-2E_2$. Note that this divisor has negative self-intersection. We know there is a union of $6$ lines that pass through each double point twice (Figure \ref{fig-phi6}), therefore $\widehat{\alpha}(I_2)\le 3$. 

Suppose there is an  effective $\Z$-divisor class $F=dH-mE_2$ such that $d/m<3$. This divisor class is $G$-invariant. Observe that 
\[ B\cdot F = (6H-2E_2)\cdot (dH-mE_2)=6d-30m<0. \]
Considering the curves $\mathcal{B}$ and $\mathcal{F}$ of classes $B$ and $F$ respectively, the curve $\mathcal{B}$ is $G$-irreducible and $\mathcal{B}\cdot \mathcal{F}<0$.  Suppose $\mathcal{B}$ has irreducible components $\mathcal{B}_1,\dots,\mathcal{B}_k$. Since $\mathcal{B}$ is $G$-irreducible, for any $i$ we have $g(\mathcal{B}_i\cdot \mathcal{F})=\mathcal{B}_j\cdot \mathcal{F}$ for some $j$. Then $\mathcal{F}$ intersects each $\mathcal{B}_i$ negatively and therefore contains $\mathcal{B}_i$. Thus the divisor class $F-B$ is effective. 

Notice that $F-B$ also intersects $B$ negatively. Performing a similar analysis, we may conclude by induction that $F-kB$ is effective for $k\ge 1$. This leads to a contradiction as $F-kB$ will eventually have negative degree. Therefore $\widehat{\alpha}(I_2)\ge 3.$

\end{proof}

We now state the Waldschmidt constants for the remaining point configurations.

\begin{proposition}\label{prop-3wald}
For the ideals $I_3$ and $I_5$, we have $\widehat{\alpha}(I_3)=3$ and  $\widehat{\alpha}(I_5)=\frac{12}{5}$.
\end{proposition}

The values $\widehat{\alpha}(I_3)$ and $\widehat{\alpha}(I_5)$ are obtained similarly. We have from Section \ref{ss-invariant}, the divisor class $6H-2E_3$ is effective. For $I_5$, there is a degree $12$ invariant vanishing to order $5$ at the quintuple points thus the divisor $12H-5E_5$. Indeed for each choice of 5 quintuple points, there is a unique conic passing through each point (Figure \ref{fig-quins}). The important property these divisors share with $B$ is that they have negative self-intersection.

\begin{figure}
    \centering
    \includegraphics[width=8cm]{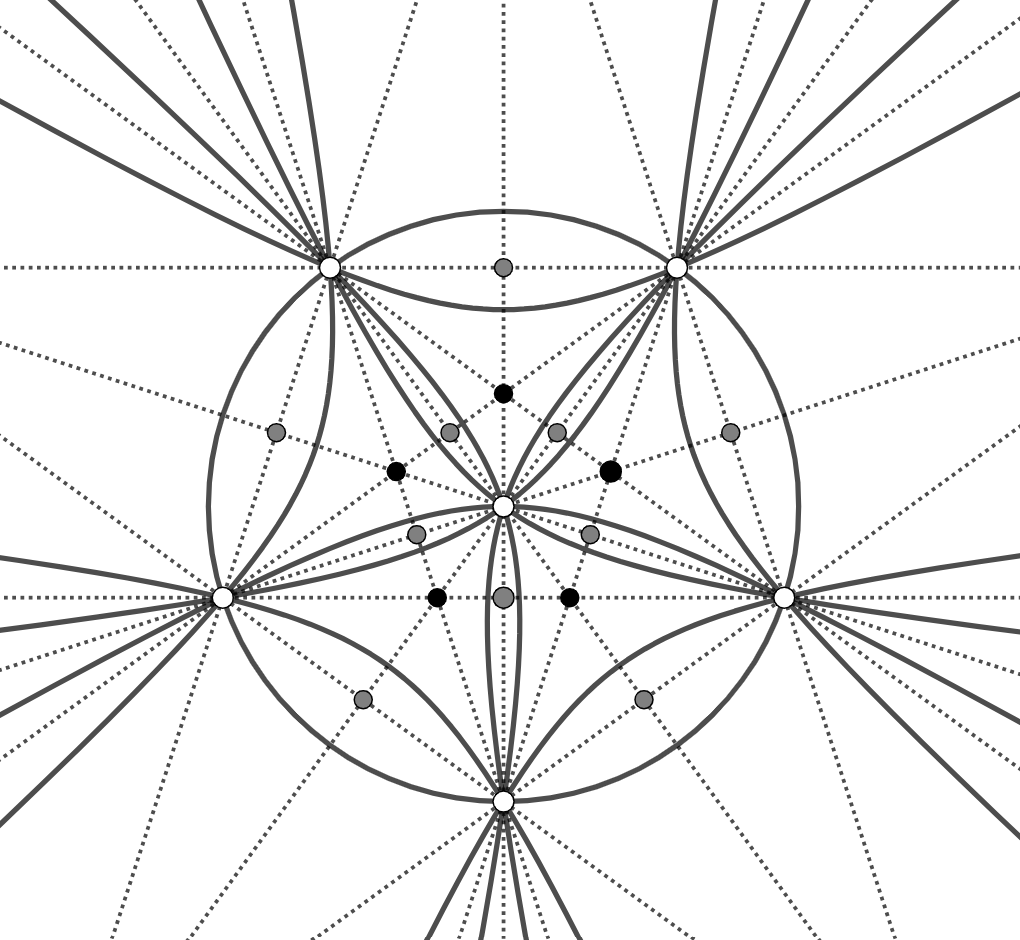}
    \caption{A curve of divisor class $12H-5E_5$ drawn solid.   }    \label{fig-quins}
\end{figure}

\bibliographystyle{plain}

\begin{thebibliography}{Cos06b}
\bibitem[AD09]{Dolgachev}Artebani, Michela, and Igor V. Dolgachev. ``The Hesse pencil of plane cubic curves." \textit{L’Enseignement Mathématique} 55, no. 3 (2009): 235-273.
    
\bibitem[BDRH19]{NCSB}Bauer, Thomas, Sandra Di Rocco, Brian Harbourne, Jack Huizenga, Alexandra Seceleanu, and Tomasz Szemberg. ``Negative curves on symmetric blowups of the projective plane, resurgences, and Waldschmidt constants." \textit{International Mathematics Research Notices} 2019, no. 24 (2019): 7459-7514. 

\bibitem[DK15]{derksen2015computational} Derksen, Harm, and Gregor Kemper. \textit{Computational invariant theory}. Springer, 2015.

\bibitem[DS21]{drabkin2021singular} Benjamin Drabkin and Alexandra Seceleanu.   ``Singular loci of reflection arrangements and the containment problem." \textit{Mathematische Zeitschrift}, 299(1):867–895, 2021.


\bibitem[Har18]{BrianAsymptotics}Harbourne, Brian. ``Asymptotics of linear systems, with connections to line arrangements." \textit{arXiv preprint arXiv:1705.09946} (2017).
    
\bibitem[Kle58]{Klein}Britton, J. L. ``Lectures on the Icosahedron and the Solution of Equations of the Fifth Degree By Felix Klein. Translated by GG Morrice. reprinted. Pp. xvi+ 289. \$1.85. 1956.(Dover Publications)." The Mathematical Gazette 42, no. 340 (1958): 139-140.

\bibitem[Nag59]{Nagata}Nagata, Masayoshi. ``On the 14-th problem of Hilbert." \textit{American Journal of Mathematics 81, no. 3} (1959): 766-772.

\bibitem[OT13]{orlik2013arrangements}Peter Orlik and Hiroaki Terao. Arrangements of hyperplanes, volume 300. Springer Science \& Business Media, 2013.

\bibitem[ST54]{shephard_todd_1954} Shephard, Geoffrey C., and John A. Todd. ``Finite unitary reflection groups." \textit{Canadian Journal of Mathematics} 6 (1954): 274-304.

    

    

    

    
\end{thebibliography}

\end{document}